\documentclass[12pt]{amsart}

\usepackage{amsthm}
\usepackage{amssymb}
\usepackage{amsmath}

\numberwithin{equation}{section}
\newtheorem{thm}{Theorem}[section]
\newtheorem{prop}[thm]{Proposition}
\newtheorem{lem}[thm]{Lemma}
\newtheorem{cor}[thm]{Corollary}

\theoremstyle{definition}

\newtheorem{rem}[thm]{Remark}

\begin{document}

\title[Quadratic relations for $q$MZV]
{Quadratic relations for \\ a $q$-analogue of multiple zeta values}
\author{Yoshihiro Takeyama}
\address{Department of Mathematics, 
Graduate School of Pure and Applied Sciences, 
Tsukuba University, Tsukuba, Ibaraki 305-8571, Japan}
\email{takeyama@math.tsukuba.ac.jp}

\begin{abstract}
We obtain a class of quadratic relations for a $q$-analogue of multiple zeta values ($q$MZV's). 
In the limit $q \to 1$, it turns into Kawashima's relation for multiple zeta values. 
As a corollary we find that $q$MZV's satisfy the linear relation contained in Kawashima's relation. 
In the proof we make use of a $q$-analogue of Newton series and 
Bradley's duality formula for finite multiple harmonic $q$-series. 
\end{abstract}
\maketitle

\setcounter{section}{0}
\setcounter{equation}{0}


\section{Introduction}

In this paper we prove quadratic relations for a $q$-analogue of 
multiple zeta values ($q$MZV's, for short). 
The relations are of a similar form to Kawashima's relation for 
multiple zeta values (MZV's). 

First let us recall the definition of $q$MZV \cite{Bnote, Z}. 
Let $\mathbf{k}=(k_{1}, \ldots , k_{r})$ be an $r$-tuple of positive integers 
such that $k_{1}\ge 2$. 
Then $q$MZV $\zeta_{q}(\mathbf{k})$ is a $q$-series defined by 
\begin{align}
\zeta_{q}(\mathbf{k}):=\sum_{m_{1}>\cdots >m_{r}>0}
\frac{q^{(k_{1}-1)m_{1}+\cdots +(k_{r}-1)m_{r}}}{[m_{1}]^{k_{1}}\cdots [m_{r}]^{k_{r}}},  
\label{eq:def-qMZV}
\end{align}
where $[n]$ is the $q$-integer 
\begin{align}
[n]:=\frac{1-q^{n}}{1-q}.  
\label{eq:def-q-integer}
\end{align}
Since $k_{1}\ge 2$, 
the right hand side of \eqref{eq:def-qMZV} is well-defined as a formal power series of $q$. 
If we regard $q$ as a complex variable, it is absolutely convergent in $|q|<1$.  
In the limit as $q \to 1$, $q$MZV turns into MZV defined by 
\begin{align*}
\zeta({\mathbf{k}}):=\sum_{m_{1}>\cdots >m_{r}>0}
\frac{1}{m_{1}^{k_{1}}\cdots m_{r}^{k_{r}}}.    
\end{align*}
An interesting point is that $q$MZV's satisfy many relations 
in the same form as those for MZV's. 
For example, $q$MZV's satisfy Ohno's relation \cite{O}, the cyclic sum formula \cite{HO, OW} and 
Ohno-Zagier's relation \cite{OZ}. 
See \cite{Bnote} for the proof of Ohno's relation and the cyclic sum formula for $q$MZV's, 
and see \cite{OT} for that of Ohno-Zagier's relation.  

In this paper we prove a class of quadratic relations for $q$MZV's 
(see Theorem \ref{thm:Kawashima-quad-modified} below). 
In the limit as $q \to 1$, it turns into Kawashima's relation \cite{Kaw} for MZV's. 
Some of our relations are linear, and 
they are completely the same as the linear part of Kawashima's relation 
(Corollary \ref{cor:Kawashima-linear}). 
It is known that the linear part of Kawashima's relation contains 
Ohno's relation \cite{Kaw} and the cyclic sum formula \cite{TW}, 
and hence we find again that $q$MZV's also satisfy them. 

The proof of our quadratic relations proceeds in a similar manner to that of Kawashima's relation. 
The ingredients are a $q$-analogue of Newton series and 
finite multiple harmonic $q$-series. 
Let $b=\{b(n)\}_{n=0}^{\infty}$ be a sequence of formal power series in $q$. 
Then we define a sequence $\nabla_{q}(b)$ by 
\begin{align*}
\nabla_{q}(b)(n):=\sum_{i=0}^{n}q^{i}\frac{(q^{-n})_{i}}{(q)_{i}}\,b(i) 
\end{align*}
and consider the series 
\begin{align*}
f_{\nabla_{q}(b)}(z):=\sum_{n=0}^{\infty}\nabla_{q}(b)(n)\,z^{n}\frac{(z^{-1})_{n}}{(q)_{n}},  
\end{align*} 
where $(x)_{n}$ is the $q$-shifted factorial 
\begin{align}
(x)_{n}:=\prod_{j=0}^{n-1}(1-xq^{j}).  
\label{eq:def-q-shifted}
\end{align}
Under some condition for $\nabla_{q}(b)$, the series $f_{\nabla_{q}(b)}(z)$ 
is well-defined as an element of $\mathbb{Q}[[q, z]]$ and 
it satisfies $f_{\nabla_{q}(b)}(q^{m})=b(m)$ for $m \ge 0$ (see Proposition \ref{prop:interpolation}). 
Thus the series $f_{\nabla_{q}(b)}(z)$ interpolates the sequence $b$, 
and can be regarded as a $q$-analogue of Newton series. 
It has a nice property:  
\begin{align}
f_{\nabla_{q}(b_{1})}f_{\nabla_{q}(b_{2})}=f_{\nabla_{q}(b_{1}b_{2})}.  
\label{eq:product-Newton}
\end{align}
Now consider the finite multiple harmonic $q$-series $S_{\mathbf{k}}(n)$ defined by 
\begin{align*}
S_{\mathbf{k}}(n):=\sum_{n \ge m_{1} \ge \cdots \ge m_{r} \ge 1}
\frac{q^{m_{1}+m_{2}+\cdots +m_{r}}}{[m_{1}]^{k_{1}} \cdots [m_{r}]^{k_{r}}}.   
\end{align*} 
Any product of $S_{\mathbf{k}}$'s can be written as a linear combination of them 
with coefficients in $\mathbb{Q}[(1-q)]$. 
The duality formula due to Bradley \cite{B} (see Proposition \ref{prod:finite-duality} below)  
implies that the coefficients are $q$MZV's in the expansion of $f_{\nabla_{q}(S_{\bf k})}$ at $z=1$.  
Therefore, by expanding 
$f_{\nabla_{q}(S_{\bf k})}f_{\nabla_{q}(S_{\mathbf{k}'})}=f_{\nabla_{q}(S_{\bf k}S_{\bf k'})}$ at $z=1$, 
we obtain quadratic relations for $q$MZV's.   

The paper is organized as follows. 
In Section \ref{sec:alg-str} we define finite multiple harmonic $q$-series 
and describe their algebraic structure by making use of a non-commutative polynomial ring. 
In Section \ref{sec:Newton} we define a $q$-analogue of Newton series and 
prove the key relation \eqref{eq:product-Newton}.  
In Section \ref{sec:Kawashima} we prove the quadratic relations and see that 
$q$MZV's satisfy the linear part of Kawashima's relation for MZV's. 

In this paper we denote by $\mathbb{N}$ the set of non-negative integers.


\section{Algebraic structure of finite multiple harmonic $q$-series}\label{sec:alg-str}

\subsection{Finite multiple harmonic $q$-series}

Let $\hbar$ be a formal variable and $\mathcal{C}:=\mathbb{Q}[\hbar]$ the coefficient ring. 
Denote by $\mathfrak{h}^{1}$ the non-commutative polynomial algebra over $\mathcal{C}$ 
freely generated by the set of alphabets $\{z_{n}\}_{n=1}^{\infty}$. 
We define the {\it depth} of a word $u=z_{i_{1}}\cdots z_{i_{r}}$ by ${\rm dep}(u):=r$. 

Let $\mathcal{R}:=\mathbb{Q}[[q]]$ be the ring of formal power series in $q$. 
We endow $\mathcal{R}$ with $\mathcal{C}$-module structure such that 
$\hbar$ acts as multiplication by $1-q$.
Denote by $\mathcal{R}^{\mathbb{N}}$ the set of sequences $b=\{b(n)\}_{n=0}^{\infty}$ 
of formal power series $b(n) \in \mathcal{R}$. 
Then $\mathcal{R}^{\mathbb{N}}$ is a $\mathcal{C}$-algebra with the product defined by 
$(bc)(n):=b(n)\,c(n)$ for $b, c\in \mathcal{R}^{\mathbb{N}}$. 

For a word $u=z_{k_{1}} \cdots z_{k_{r}}$,  
we define $S_u, A_u, A^{\star}_u \in \mathcal{R}^{\mathbb{N}}$ by 
\begin{align}
S_{u}(n)&:=
\sum_{n \ge m_{1} \ge \cdots \ge m_{r} \ge 1}
\frac{q^{m_{1}+m_{2}+\cdots +m_{r}}}{[m_{1}]^{k_{1}} \cdots [m_{r}]^{k_{r}}}, 
\label{eq:def-S} \\ 
A_{u}(n)&:=
\sum_{n \ge m_{1} >\cdots >m_{r}>0}
\frac{q^{(k_{1}-1)m_{1}+(k_{2}-1)m_{2}+\cdots +(k_{r}-1)m_{r}}}
{[m_{1}]^{k_{1}} \cdots [m_{r}]^{k_{r}}}, 
\label{eq:def-A} \\ 
A^{\star}_{u}(n)&:=
\sum_{n \ge m_{1} \ge \cdots \ge m_{r}\ge 1}
\frac{q^{(k_{1}-1)m_{1}+(k_{2}-1)m_{2}+\cdots +(k_{r}-1)m_{r}}}
{[m_{1}]^{k_{1}} \cdots [m_{r}]^{k_{r}}},    
\end{align}
where $[n]$ is the $q$-integer defined by \eqref{eq:def-q-integer}. 
Setting $S_{1}(n), A_{1}(n), A^{\star}_{1}(n) \equiv 1$, 
we extend the correspondence $u \mapsto S_{u}, A_{u}, A^{\star}_{u}$ to 
the $\mathcal{C}$-linear map $S, A, A^{\star} \, : \mathfrak{h}^{1} \to \mathcal{R}^{\mathbb{N}}$. 

Let $\mathfrak{h}_{>0}^{1}$ be the $\mathcal{C}$-submodule consisting of non-constant elements, that is, 
$\mathfrak{h}_{>0}^{1}=\sum_{k=1}^{\infty}
\sum_{i_{1}, \ldots , i_{k}\ge 1}\mathcal{C}z_{i_{1}}\cdots z_{i_{k}}$. 
For a word $u=z_{k_{1}} \cdots z_{k_{r}} \in \mathfrak{h}_{>0}^{1}$, set 
$s_u, a_u \in \mathcal{R}^{\mathbb{N}}$ by 
\begin{align}
s_{u}(n)&:=\sum_{n+1=m_{1} \ge m_{2} \ge \cdots \ge m_{r} \ge 1}
\frac{q^{k_{1}m_{1}+(k_{2}-1)m_{2}+\cdots +(k_{r}-1)m_{r}}}
{[m_{1}]^{k_{1}} \cdots [m_{r}]^{k_{r}}}, \\  
a_{u}(n)&:=\sum_{n+1=m_{1}>m_{2}> \cdots >m_{r}>0}
\frac{q^{(k_{1}-1)m_{1}+(k_{2}-1)m_{2}+\cdots +(k_{r}-1)m_{r}}}
{[m_{1}]^{k_{1}} \cdots [m_{r}]^{k_{r}}}. 
\label{eq:def-a}
\end{align}
Extending by $\mathcal{C}$-linearity 
we define two maps $s, a: \mathfrak{h}_{>0}^{1} \to \mathcal{R}^{\mathbb{N}}$. 
Note that if $w=z_{i}w' \, (w, w' \in \mathfrak{h}^{1}, i \ge 1)$ we have 
\begin{align}
s_{w}(n)=\frac{q^{i(n+1)}}{[n+1]^{i}}\,A^{\star}_{w'}(n+1), \qquad 
a_{w}(n)=\frac{q^{(i-1)(n+1)}}{[n+1]^{i}}\,A_{w'}(n). 
\label{eq:sa-to-A}
\end{align}

We define a map $\phi: \mathfrak{h}^{1} \to \mathfrak{h}^{1}$ as follows. 
For a word $u=z_{k_{1}} \cdots z_{k_{r}}$, 
consider the set $I_{u}=\{\sum_{i=1}^{j}k_{i} \, |\, 1\le j <r\}$. 
Denote by $I_{u}^{c}$ the set consisting of positive integers 
which are less than or equal to $\sum_{i=1}^{r}k_{i}$ not belonging to $I_{u}$. 
Set $I_{u}^{c}=\{p_{1}, \ldots , p_{l}\} \, (p_{1}<\cdots <p_{l})$ 
and define $\phi(u):=z_{k'_{1}}\cdots z_{k'_{l}}$, where 
$k'_{1}=p_{1}, \, k'_{i}=p_{i}-p_{i-1}\, (2 \le i \le l)$. 
Extending by $\mathcal{C}$-linearity $\phi$ is defined as a map on $\mathfrak{h}^{1}$. 
Note that $\phi^{2}={\rm id}$. 

Now we define a $\mathcal{C}$-linear map 
$\nabla_{q} : \mathcal{R}^{\mathbb{N}} \to \mathcal{R}^{\mathbb{N}}$ by 
\begin{align*}
\nabla_{q}(b)(n):=\sum_{i=0}^{n}q^{i}\frac{(q^{-n})_{i}}{(q)_{i}}\,b(i), 
\end{align*}
where $(x)_{n}$ is the $q$-shifted factorial defined by \eqref{eq:def-q-shifted}.
The following duality formula is due to Bradley \cite{B} (see also \cite{Kaw-q}): 

\begin{prop}\label{prod:finite-duality}
For $w \in \mathfrak{h}_{>0}^{1}$, we have 
\begin{align*}
\nabla_{q}(S_{w})(n)=\left\{ 
\begin{array}{ll}
0 & (n=0), \\
{}-s_{\phi(w)}(n-1) & (n\ge 1). 
\end{array}
\right.
\end{align*}
\end{prop}

\begin{rem}
The operator $\nabla_{q}$ has another description. 
Consider the difference operator 
$\Delta_{t}\,:\mathcal{R}^{\mathbb{N}} \to \mathcal{R}^{\mathbb{N}}$ defined by 
$\Delta_{t}(b)(n):=b(n)-t b(n+1)$. 
Then we have 
$\nabla_{q}(b)(n)=(\Delta_{q^{-(n-1)}} \circ \cdots \circ \Delta_{q^{-1}} \circ \Delta_{1}(b))(0)$ 
(see Corollary 2.7 in \cite{Kaw-q}). 
\end{rem}

\subsection{A $q$-analogue of multiple zeta (star) values}

Introduce the valuation $v: \mathcal{R} \to \mathbb{Z}_{\ge 0}\cup\{+\infty\}$ 
defined by $v(f):=\inf\{j \,|\, c_{j}\not=0\}$ for $f=\sum_{j=0}^{\infty}c_{j}q^{j}$,   
and endow $\mathcal{R}$ with the topology such that 
the sets $\{f+\phi\,|\, v(\phi)\ge n\} \, (n=0, 1, \ldots)$ form neighborhood base 
at $f \in \mathcal{R}$. 
Then $\mathcal{R}$ is complete. 

Denote by $\mathfrak{h}^{0}$ the $\mathcal{C}$-submodule of $\mathfrak{h}^{1}$ 
generated by $1$ and the words $z_{k_{1}} \cdots z_{k_{r}}$ satisfying $r \ge 1$ and $k_{1} \ge 2$. 
If $w \in \mathfrak{h}^{0}$, 
the limit $\lim_{n \to \infty}A_{w}(n)=\sum_{k=0}^{\infty}a_{w}(k)$ converges in 
$\mathcal{R}$ since $v(a_{w}(n))\ge n+1$. 
We call it a {\it $q$-analogue of multiple zeta value} ($q$MZV) and denote it by $\zeta_{q}(w)$.  
If $w$ is a word $w=z_{k_{1}} \cdots z_{k_{r}} \in \mathfrak{h}^{0}$, 
it is given by 
\begin{align*}
\zeta_{q}(z_{k_{1}} \cdots z_{k_{r}}):=\sum_{m_{1} >\cdots >m_{r}>0}
\frac{q^{(k_{1}-1)m_{1}+(k_{2}-1)m_{2}+\cdots +(k_{r}-1)m_{r}}}
{[m_{1}]^{k_{1}} \cdots [m_{r}]^{k_{r}}}. 
\end{align*}
We also define a {\it $q$-analogue of multiple zeta star value} ($q$MZSV) 
(or that of {\it non-strict multiple zeta value}) by 
\begin{align*}
\zeta_{q}^{\star}(z_{k_{1}} \cdots z_{k_{r}}):=\sum_{m_{1} \ge \cdots \ge m_{r}\ge 1}
\frac{q^{(k_{1}-1)m_{1}+(k_{2}-1)m_{2}+\cdots +(k_{r}-1)m_{r}}}
{[m_{1}]^{k_{1}} \cdots [m_{r}]^{k_{r}}}.   
\end{align*}
If we regard $q$ as a complex variable, $q$MZV and $q$MZSV 
are absolutely convergent in $|q|<1$. 
Therefore in fact they are convergent series. 
In the limit $q \to 1$, they converge to multiple zeta values (MZV) 
and multiple zeta star values (or non-strict multiple zeta values) defined by 
\begin{align*}
\zeta(z_{k_{1}} \cdots z_{k_{r}}):=\sum_{m_{1} >\cdots >m_{r}>0}
\frac{1}{m_{1}^{k_{1}} \cdots m_{r}^{k_{r}}} 
\end{align*}
and 
\begin{align*}
\zeta^{\star}(z_{k_{1}} \cdots z_{k_{r}}):=\sum_{m_{1} \ge \cdots \ge m_{r}\ge 1}
\frac{1}{m_{1}^{k_{1}} \cdots m_{r}^{k_{r}}},    
\end{align*}
respectively.  

\subsection{Algebraic structure}

Let $\mathfrak{z}$ be the $\mathcal{C}$-submodule of $\mathfrak{h}^{1}$ generated by $\{z_{i}\}_{i=1}^{\infty}$. 
We define three products $\circ, \circ_{\pm}$ on $\mathfrak{z}$ by setting 
\begin{align*}
z_{i}\circ z_{j}:=z_{i+j}, \qquad 
z_{i}\circ_{\pm} z_{j}:=\pm z_{i+j}+\hbar z_{i+j-1} 
\end{align*}
and extending by $\mathcal{C}$-linearity. 
These products are associative and commutative. 

Define two $\mathcal{C}$-bilinear products $*_{\pm}$ on $\mathfrak{h}^{1}$ inductively by 
\begin{align}
& 
1*_{\pm} w=w, \quad w*_{\pm} 1=w, 
\nonumber \\ 
& 
(z_{i}w_{1})*_{\pm}(z_{j}w_{2})=
z_{i}(w_{1}*_{\pm}z_{j}w_{2})+z_{j}(z_{i}w_{1}*_{\pm}w_{2})+(z_{i}\circ_{\pm}z_{j})(w_{1}*_{\pm}w_{2}) 
\label{eq:prod-recurrence}
\end{align}
for $i, j\ge 1$ and $w, w_{1}, w_{2} \in \mathfrak{h}^{1}$. 
These products are commutative. 

\begin{prop}\label{prop:prod-S}
Let $w_{1}, w_{2} \in \mathfrak{h}^{1}$. Then 
$S_{w_{1}}S_{w_{2}}=S_{w_{1}*_{-}w_{2}}$ and $A_{w_{1}}A_{w_{2}}=A_{w_{1}*_{+}w_{2}}$. 
\end{prop}

\begin{proof}
It suffices to show that 
$S_{w_{1}}S_{w_{2}}$ and $A_{w_{1}}A_{w_{2}}$ 
follow the recurrence relation \eqref{eq:prod-recurrence} for $*_{-}$ and $*_{+}$, respectively. 
It can be checked by using 
\begin{align*}
\frac{q^{2m}}{[m]^{i+j}}=\frac{q^{m}}{[m]^{i+j}}-(1-q)\frac{q^{m}}{[m]^{i+j-1}} 
\end{align*}
and 
\begin{align}
\frac{q^{(i+j-2)m}}{[m]^{i+j}}=\frac{q^{(i+j-1)m}}{[m]^{i+j}}+(1-q)\frac{q^{(i+j-2)m}}{[m]^{i+j-1}}  
\label{eq:separation}
\end{align}
for $i, j \ge 1$. 
\end{proof}

The products $\circ$ and $\circ_{+}$ determine $\mathfrak{z}$-module structures on $\mathfrak{h}^{1}$, 
which we denote by the same letters $\circ$ and $\circ_{+}$, such that 
\begin{align*}
z_{i}\circ 1:=0, \qquad z_{i} \circ (z_{j}w):=(z_{i} \circ z_{j})w \quad 
(i, j\ge 1, \, w \in \mathfrak{h}^{1}) 
\end{align*}
and the above formulas where $\circ$ is replaced with $\circ_{\pm}$. 
By convention we set $z_{0}\circ c=0$ for $c \in \mathcal{C}$ and 
$z_{0} \circ w=w$ for $w \in \mathfrak{h}_{>0}^{1}$. 
Then $z_{i}\circ_{+} w=(z_{i}+\hbar z_{i-1})\circ w$ for 
$i\ge 1$ and $w \in \mathfrak{h}^{1}$. 
 
Consider the $\mathcal{C}$-linear map $d_{q}$ on $\mathfrak{h}^{1}$ defined inductively by 
\begin{align*}
d_{q}(1)=1, \qquad 
d_{q}(z_{i}w)=z_{i}\,d_{q}(w)+z_{i}\circ_{+}d_{q}(w) \quad (i\ge 1, \, w \in \mathfrak{h}^{1}).  
\end{align*}
Note that $d_{q}$ is invertible and its inverse satisfies 
$d_{q}^{-1}(1)=1$ and 
\begin{align*}
d_{q}^{-1}(z_{i}w)=z_{i}\,d_{q}^{-1}(w)-z_{i}\circ_{+}d_{q}^{-1}(w) \quad 
(i\ge 1, \, w \in \mathfrak{h}^{1}).   
\end{align*}

\begin{lem}\label{lem:action-commute}
For $i \ge 1$ and $w \in \mathfrak{h}^{1}$, we have $d_{q}(z_{i}\circ w)=z_{i} \circ d_{q}(w)$, 
that is, the map $d_{q}$ commutes with the action $\circ$ of $\mathfrak{z}$.  
\end{lem}

\begin{proof}
For $w=1$ it is trivial. 
Suppose that $w$ is a word and set $w=z_{j}w'$. 
Then we have 
$d_{q}(z_{i}\circ w)=d_{q}(z_{i+j}w')=z_{i+j}\,d_{q}(w')+z_{i+j}\circ_{+}d_{q}(w')=
 z_{i}\circ(z_{j}d_{q}(w'))+(z_{i}\circ z_{j})\circ_{+}d_{q}(w')$.
{}Using $(z_{i} \circ z_{j}) \circ_{+} z_{k}=z_{i} \circ (z_{j} \circ_{+} z_{k})$ for $i, j, k \ge 1$, 
we see that $(z_{i}\circ z_{j})\circ_{+}d_{q}(w')=z_{i}\circ(z_{j}\circ_{+}d_{q}(w'))$.  
Thus we get $d_{q}(z_{i}\circ w)=z_{i}\circ d_{q}(w)$ for a word $w$. 
{}From the $\mathcal{C}$-linearity of $d_{q}$, we obtain the lemma. 
\end{proof}

\begin{prop}\label{prop:to-star}
Let $w \in  \mathfrak{h}^{1}$. Then $A^{\star}_{w}=A_{d_{q}(w)}$.  
\end{prop}

\begin{proof}
It is enough to prove the case where $w$ is a word. 
We prove the proposition by induction on the depth of $w$. 
If ${\rm dep}(w)=1$, that is, $w=z_{i}$ for some $i\ge1$, 
we have $A^{\star}_{z_{i}}=A_{z_{i}}=A_{d_{q}(z_{i})}$. 

{}From the definition of $A$ and $A^{\star}$ we find that 
\begin{align*}
\sum_{m=1}^{n}\frac{q^{(i-1)m}}{[m]^{i}}A_{w}^{\star}(m)&=A_{z_{i}w}^{\star}(n), \\ 
\sum_{m=1}^{n}\frac{q^{(i-1)m}}{[m]^{i}}A_{w}(m)&=A_{z_{i}w+z_{i}\circ_{+}w}(n)  
\end{align*}
for $w \in \mathfrak{h}^{1}$. 
To show the second formula, divide the sum in the definition \eqref{eq:def-A} of $A_{w}(m)$ into two parts 
with $m_{1}=m$ and with $m_{1}<m$, and use \eqref{eq:separation} for the first part.   
Now suppose that ${\rm dep}(w)>1$ and 
the proposition holds for words whose depth is less than ${\rm dep}(w)$. 
Then setting $w=z_{i}w'$, we find 
\begin{align*}
A^{\star}_{w}(n)&=\sum_{m=1}^{n}\frac{q^{(i-1)m}}{[m]^{i}}A^{\star}_{w'}(m)=\sum_{m=1}^{n}
\frac{q^{(i-1)m}}{[m]^{i}}A_{d_{q}(w')}(m) \\ 
&=A_{z_{i}d_{q}(w')+z_{i}\circ_{+}d_{q}(w')}(n)=A_{d_{q}(w)}(n). 
\end{align*}
\end{proof}

\begin{cor}\label{cor:to-star}
For $w \in  \mathfrak{h}^{1}$, we have $\zeta_{q}^{\star}(w)=\zeta_{q}(d_{q}(w))$.   
\end{cor}

We define a $\mathcal{C}$-bilinear product $\circledast_{q}$ on $\mathfrak{h}_{>0}^{1}$ by 
\begin{align}
(z_{i}w_{1})\circledast_{q}(z_{j}w_{2}):=(z_{i} \circ z_{j})(w_{1}*_{+}w_{2}) \qquad 
(i, j\ge 1, \, w_{1}, w_{2} \in \mathfrak{h}^{1}).  
\label{eq:def-circleprod}
\end{align}
It is commutative and associative. 

\begin{prop}\label{prop:prod-sa}
Let $w_{1}, w_{2} \in \mathfrak{h}_{>0}^{1}$. Then 
$s_{w_{1}}a_{w_{2}}=a_{d_{q}(w_{1})\circledast_{q} w_{2}}$.  
\end{prop}

To prove Proposition \ref{prop:prod-sa} we need the following lemma: 
\begin{lem}\label{lem:A-to-a}
For $n\in \mathbb{N}$ and $w \in \mathfrak{h}^{1}$, we have 
\begin{align}
\frac{q^{(i-1)(n+1)}}{[n+1]^{i}}\,A_{w}(n+1)=a_{z_{i}w+z_{i}\circ_{+}w}(n).  
\label{eq:A-to-a}
\end{align} 
\end{lem}

\begin{proof}
It is trivial for $w=1$. 
Without loss of generality we can assume that $w$ is a word. 
Set $w=z_{j}w'$. Then we have 
\begin{align*}
A_{w}(n+1)=\frac{q^{(j-1)(n+1)}}{[n+1]^{j}}A_{w'}(n)+A_{w}(n).  
\end{align*}
Substitute it into the left hand side of \eqref{eq:A-to-a} and use \eqref{eq:separation}.  
Then we obtain 
\begin{align*}
\left(\frac{q^{(i+j-1)(n+1)}}{[n+1]^{i+j}}+(1-q)\frac{q^{(i+j-2)(n+1)}}{[n+1]^{i+j-1}}\right)A_{w'}(n)+
\frac{q^{(i-1)(n+1)}}{[n+1]^{i}}\,A_{w'}(n). 
\end{align*}
Using \eqref{eq:sa-to-A} again, we see that it is equal to the right hand side of \eqref{eq:A-to-a}. 
\end{proof}

\begin{proof}[Proof of Proposition \ref{prop:prod-sa}]
We can assume that $w_{1}$ and $w_{2}$ are words. 
Set $w_{1}=z_{i}w_{1}'$ and $w_{2}=z_{j}w_{2}'$. 
Using \eqref{eq:sa-to-A} and Proposition \ref{prop:to-star}, we have 
\begin{align*}
(s_{w_{1}}a_{w_{2}})(n)=\frac{q^{(i+j-1)(n+1)}}{[n+1]^{i+j}}
A_{d_{q}(w_{1}')}(n+1)A_{w_{2}'}(n).  
\end{align*}
It is equal to $a_{d_{q}(z_{i+j}w_{1}')}(n)A_{w_{2}'}(n)$ because of 
Lemma \ref{lem:A-to-a} and the definition of $d_{q}$. 
Now introduce a $\mathcal{C}$-bilinear map 
$\triangle: \mathfrak{h}_{>0}^{1}\times\mathfrak{h}^{1}\to \mathfrak{h}_{>0}^{1}$ 
uniquely determined from the property  
\begin{align*}
(z_{i}w') \triangle w'':=z_{i}(w'*_{+}w'') \qquad (i\ge 1).   
\end{align*}
Then $a_{w}A_{w''}=a_{w \triangle w''}$ 
for $w \in \mathfrak{h}_{>0}^{1}$ and $w'' \in \mathfrak{h}^{1}$ 
because of \eqref{eq:sa-to-A} and Proposition \ref{prop:prod-S}. 
Therefore Proposition \ref{prop:prod-sa} is reduced to the identity 
\begin{align}
\left(d_{q}(z_{i+j}w_{1}')\right)\triangle w_{2}'=\left(d_{q}(z_{i}w_{1}')\right)\circledast_{q}(z_{j}w_{2}') 
\label{eq:sa-final}
\end{align}
for $i, j\ge 1$ and $w_{1}', w_{2}' \in \mathfrak{h}^{1}$. 

Let us prove \eqref{eq:sa-final}. 
{}From the definition of $d_{q}$ and $\triangle$, the left hand side is equal to 
\begin{align*}
z_{i+j}(d_{q}(w_{1}')*_{+}w_{2}')+(z_{i+j}\circ_{+}d_{q}(w_{1}'))\triangle w_{2}'.  
\end{align*}
The first term is equal to $(z_{i}d_{q}(w_{1}'))\circledast_{q}(z_{j}w_{2}')$ because $z_{i+j}=z_{i}\circ z_{j}$. 
Now note that $z_{i+j}\circ_{+}z_{k}=(z_{i}\circ_{+}z_{k})\circ z_{j}$ for any $i, j, k\ge 1$. 
Using this we see that the second term is equal to $(z_{i}\circ_{+}d_{q}(w_{1}'))\circledast_{q}(z_{j}w_{2}')$. 
Summing up these two terms we get the right hand side of \eqref{eq:sa-final}. 
\end{proof}


\section{A $q$-analogue of Newton series}\label{sec:Newton}

We consider the ring of formal power series $\mathcal{R}[[z]]=\mathbb{Q}[[q, z]]$. 
Introduce the valuation on $\mathcal{R}[[z]]$ by $v(f)=\inf\{i+j \,|\, c_{ij}\not=0 \}$ for 
$f=\sum_{i, j=0}^{\infty}c_{ij}q^{i}z^{j}$ and 
endow $\mathcal{R}[[z]]$ with the structure of topological $\mathbb{Q}$-algebra.  

Define polynomials $B_{n}(z) \, (n\in\mathbb{N})$ by 
\begin{align*}
B_{0}(z):=1, \qquad 
B_{n}(z):=z^{n}\frac{(z^{-1})_{n}}{(q)_{n}}=\prod_{j=1}^{n}\frac{z-q^{j-1}}{1-q^{j}} \quad (n \ge 1).  
\end{align*}

\begin{prop}\label{prop:interpolation}
Let $b \in \mathcal{R}^{\mathbb{N}}$ and $m, l \in \mathbb{N}$. Then 
\begin{align}
\sum_{n=0}^{m}\nabla(b)(l+n)\,B_{n}(q^{m})=q^{-lm}
\sum_{j=0}^{l}q^{j}\frac{(q^{-l})_{j}}{(q)_{j}}\,b(j+m).  
\label{eq:general-interpolation}
\end{align} 
In particular, we have $\sum_{n=0}^{m}\nabla(b)(n)B_{n}(q^{m})=b(m)$. 
\end{prop}

In the proof of Proposition \ref{prop:interpolation} we use the following formula: 
\begin{lem}\label{lem:key-sum}
For $m, j \in \mathbb{N}$ we have 
\begin{align*}
\frac{1}{(q)_{j}}\sum_{n=0}^{m}q^{mn}\frac{(q^{-m})_{n}(tq^{-n})_{j}}{(q)_{n}}=\left\{ 
\begin{array}{ll}
0 & (0 \le j<m),  \\
\displaystyle (q^{-1}t)^{m}\frac{(t)_{j-m}}{(q)_{j-m}} & (m \le j). 
\end{array}
\right. 
\end{align*} 
\end{lem}

\begin{proof}
Multiply the both sides by $s^{j}$ and sum up over $j \in \mathbb{N}$ 
using the $q$-binomial formula 
\begin{align}
\sum_{n=0}^{\infty}\frac{(a)_{n}}{(q)_{n}}x^{n}=\frac{(ax)_{\infty}}{(x)_{\infty}}.  
\label{eq:q-binomial}
\end{align}
Then we see that the identity to prove is equivalent to 
\begin{align}
\sum_{n=0}^{m}q^{mn}\frac{(q^{-m})_{n}}{(q)_{n}}(q^{-n}st)_{n}=(q^{-1}st)^{m}.  
\label{eq:to-show-interpolation}
\end{align} 
Using 
\begin{align}
(q^{-a}x)_{n}=(-x)^{n}q^{-an+n(n-1)/2}(q^{a-n+1}x^{-1})_{n},  
\label{eq:q-inverse}
\end{align}
we find that the left hand side of \eqref{eq:to-show-interpolation} is equal to 
\begin{align*}
(q)_{m}\sum_{n=0}^{m}\frac{(qs^{-1}t^{-1})_{n}}{(q)_{n}(q)_{m-n}}(q^{-1}st)^{n}.  
\end{align*}
It is equal to the coefficient of $x^{m}$ in 
\begin{align*}
(q)_{m}\left(\sum_{i=0}^{\infty}(q^{-1}stx)^{i}\frac{(qs^{-1}t^{-1})_{i}}{(q)_{i}}\right)
\left(\sum_{i=0}^{\infty}\frac{x^{i}}{(q)_{i}}\right),  
\end{align*}
and the above product is equal to 
\begin{align*}
(q)_{m}\frac{(x)_{\infty}}{(q^{-1}stx)_{\infty}}\frac{1}{(x)_{\infty}}=(q)_{m}
\frac{1}{(q^{-1}stx)_{\infty}}=(q)_{m}\sum_{j=0}^{\infty}\frac{(q^{-1}stx)^{j}}{(q)_{j}}  
\end{align*}
{}from the $q$-binomial formula \eqref{eq:q-binomial}. 
Thus we obtain the right hand side of \eqref{eq:to-show-interpolation}. 
\end{proof}

\begin{proof}[Proof of Proposition \ref{prop:interpolation}]
From $(q^{-l-n})_{j}=0$ for $j>l+n$, we have 
\begin{align*}
\sum_{n=0}^{m}\nabla(b)(l+n)\,B_{n}(q^{m})=\sum_{j=0}^{l+m}\frac{q^{j}}{(q)_{j}}\,b(j) 
\sum_{n=0}^{m}q^{mn}\frac{(q^{-m})_{n}(q^{-l-n})_{j}}{(q)_{n}}.  
\end{align*}
The second sum in the right hand side can be calculated 
by using Lemma \ref{lem:key-sum} with $t=q^{-l}$. 
Then we obtain the right hand side of \eqref{eq:general-interpolation}. 
\end{proof}

Suppose that $c \in \mathcal{R}^{\mathbb{N}}$ satisfies 
\begin{align}
v(c(n))\ge n \quad \hbox{for all $n \in \mathbb{N}$}. 
\label{eq:conv-cond}
\end{align}
Set 
\begin{align}
f_{c}(z):=\sum_{n=0}^{\infty}c(n)\,B_{n}(z).  
\label{eq:Newton}
\end{align}
Then the series $f_{c}(z)$ converges in $\mathcal{R}[[z]]$. 
If $c=\nabla_{q}(b)$ satisfies the condition \eqref{eq:conv-cond},  
the series $f_{\nabla_{q}(b)}$ is well-defined and 
we have $f_{\nabla_{q}(b)}(q^{m})=b(m)$ for all $m \in \mathbb{N}$ from Lemma \ref{prop:interpolation}. 
Thus we may regard the series $f_{\nabla_{q}(b)}(z)$ 
as a $q$-analogue of Newton series interpolating $b \in \mathcal{R}^{\mathbb{N}}$. 

Let us prove two properties of the series $f_{c}$. 

\begin{prop}\label{prop:Newton-expand}
Suppose that $c \in \mathcal{R}^{\mathbb{N}}$ satisfies the condition \eqref{eq:conv-cond}. 
Then the following equality holds in $\mathcal{R}[[z]]$: 
\begin{align*}
f_{c}(z)=c(0)+\sum_{m=1}^{\infty}\left(\sum_{n=1}^{\infty}c(n)\,a_{z_{1}^{m}}(n-1)\right)
\left(\frac{z-1}{1-q}\right)^{m},   
\end{align*}
where $a_{z_{1}^{m}}(n)$ is defined by \eqref{eq:def-a}. 
\end{prop}

\begin{proof}
It follows from 
\begin{align*}
B_{n}(z)=\frac{1}{[n]}\frac{z-1}{1-q}\prod_{j=1}^{n-1}\left(1+\frac{1}{[j]}\frac{z-1}{1-q}\right)=
\sum_{m=1}^{n}a_{z_{1}^{m}}(n-1) \left(\frac{z-1}{1-q}\right)^{m}
\end{align*} 
and $a_{z_{1}^{m}}(k-1)=0$ for $k<m$. 
\end{proof}

\begin{prop}\label{prop:Newton-prod}
Suppose that $\nabla_{q}(b_{i}) \in \mathcal{R}^{\mathbb{N}} \, (i=1, 2)$ satisfy 
the condition \eqref{eq:conv-cond}. 
Then we have $f_{\nabla_{q}(b_{1})}(z)f_{\nabla_{q}(b_{2})}(z)=f_{\nabla_{q}(b_{1}b_{2})}(z)$. 
\end{prop}

We prove two lemmas to show Proposition \ref{prop:Newton-prod}. 

\begin{lem}
For $n \in \mathbb{N}$ we have 
\begin{align}
B_{n}(z)=y^{n}\sum_{j=0}^{n}\frac{(y^{-1})_{n-j}}{(q)_{n-j}}B_{j}(y^{-1}z).  
\label{eq:B-connect}
\end{align} 
\end{lem}

\begin{proof}
The right hand side is equal to the coefficient of $t^{n}$ in 
\begin{align*}
\left(\sum_{j=0}^{\infty}(ty)^{j}\frac{(y^{-1})_{j}}{(q)_{j}}\right)
\left(\sum_{j=0}^{\infty}(tz)^{j}\frac{(yz^{-1})_{j}}{(q)_{j}}\right).  
\end{align*} 
{}From the $q$-binomial formula \eqref{eq:q-binomial}, the above product is equal to 
\begin{align*}
\frac{(t)_{\infty}}{(yt)_{\infty}}\frac{(yt)_{\infty}}{(tz)_{\infty}}=
\frac{(t)_{\infty}}{(tz)_{\infty}}=\sum_{n=0}^{\infty}t^{n}B_{n}(z). 
\end{align*}
This completes the proof. 
\end{proof}

\begin{lem}\label{lem:Newton-prod}
Suppose that $c_{1}, c_{2} \in \mathcal{R}^{\mathbb{N}}$ satisfy the condition \eqref{eq:conv-cond}. 
Set $c_{3} \in \mathcal{R}^{\mathbb{N}}$ by 
\begin{align*}
c_{3}(n):=\sum_{k=0}^{n}{n \brack k}_{q}c_{1}(k)\,
\sum_{j=0}^{k}c_{2}(n-k+j)B_{j}(q^{k}), 
\end{align*} 
where ${n \brack k}_{q}$ is the $q$-binomial coefficient 
\begin{align*}
{n \brack k}_{q}:=\frac{(q)_{n}}{(q)_{k}(q)_{n-k}}. 
\end{align*}
Then we have $f_{c_{1}}(z)f_{c_{2}}(z)=f_{c_{3}}(z)$. 
\end{lem}

\begin{proof}
Substitute \eqref{eq:B-connect} into 
$f_{c_{1}}(z)f_{c_{2}}(z)=\sum_{m, n=0}^{\infty}c_{1}(m)c_{2}(n)B_{m}(z)B_{n}(z)$ 
setting $y=q^{m}$. 
Using $B_{m}(z)B_{j}(q^{-m}z)=q^{-jm}B_{m+j}(z)$, we find that
\begin{align*}
f_{c_{1}}(z)f_{c_{2}}(z)=\sum_{m, n=0}^{\infty}\sum_{j=0}^{n}c_{1}(m)c_{2}(n) 
q^{(n-j)m}\frac{(q^{-m})_{n-j}}{(q)_{n-j}}{m+j \brack m}_{q}B_{m+j}(z).  
\end{align*} 
Then we can see that it is equal to $f_{c_{3}}(z)$. 
\end{proof}

\begin{proof}[Proof of Proposition \ref{prop:Newton-prod}]
{}From Proposition \ref{prop:interpolation} and Lemma \ref{lem:Newton-prod}, 
it suffices to prove that 
\begin{align}
\sum_{0\le i\le j \le n}\frac{q^{i+j}}{(q)_{i}}\,b_{1}(i)\,b_{2}(j)
\sum_{k=i}^{j}q^{k(k-n-1)}{n \brack k}_{q}\frac{(q^{-k})_{i}(q^{-n+k})_{j-k}}{(q)_{j-k}}=
\nabla(b_{1}b_{2})(n). 
\label{eq:f-prod-coeff}
\end{align} 
Using \eqref{eq:q-inverse}, we see that 
the second sum in the left hand side above is equal to 
\begin{align*}
(-1)^{j}q^{j(j-1)/2-nj-i}\frac{(q)_{n}}{(q)_{n-j}(q)_{j-i}} 
\sum_{k=0}^{j-i}(-1)^{k}q^{k(k-1)/2}{j-i \brack k}_{q}. 
\end{align*}
For $N \in \mathbb{N}$, we have 
\begin{align*}
\sum_{k=0}^{N}(-1)^{k}q^{k(k-1)/2}{N \brack k}_{q}=\delta_{N, 0}.  
\end{align*}
Thus we find that the left hand side of \eqref{eq:f-prod-coeff} is equal to 
\begin{align*}
\sum_{i=0}^{n}(-1)^{i}\frac{q^{-ni+i(i+1)/2}}{(q)_{i}}\frac{(q)_{n}}{(q)_{n-i}}\,b_{1}(i)\,b_{2}(i)=
\sum_{i=0}^{n}q^{i}\frac{(q^{-n})_{i}}{(q)_{i}}(b_{1}b_{2})(i)=\nabla(b_{1}b_{2})(n).  
\end{align*}
\end{proof}


\section{A $q$-analogue of Kawashima's relation}\label{sec:Kawashima}

\subsection{Kawashima's relation}

Define a product $\overline{*}$ on $\mathfrak{h}^{1}$ inductively by 
\begin{align*}
& 
1\,\overline{*}\,w=w, \quad w\,\overline{*}\,1=w, \\ 
& 
(z_{i}w_{1})\,\overline{*}\,(z_{j}w_{2})=
z_{i}(w_{1}\,\overline{*}\,z_{j}w_{2})+z_{j}(z_{i}w_{1}\,\overline{*}\,w_{2})-z_{i+j}
(w_{1}\,\overline{*}\,w_{2})  
\end{align*}
for $i, j \ge 1$ and $w, w_{1}, w_{2} \in \mathfrak{h}^{1}$. 
We also define a product $\circledast$ on $\mathfrak{h}_{>0}^{1}$ by 
the formula \eqref{eq:def-circleprod} where the product 
$*_{+}$ in the right hand side is replaced with $*$. 
We define a $\mathcal{C}$-linear map $d: \mathfrak{h}^{1} \to \mathfrak{h}^{1}$ by 
the properties $d(1)=1$ and $d(z_{i}w)=z_{i}\,d(w)+z_{i}\circ d(w)$ for $w \in \mathfrak{h}^{1}$. 
The products $\overline{*}, \circledast$ and the map $d$ can be regarded as the limit of 
$*_{-}, \circledast_{q}$ and $d_{q}$ as $\hbar\to 0$ (or $q \to 1$), respectively.  

Kawashima proved the following relation for MZV's \cite{Kaw}:  
\begin{align}
\zeta(d(\phi(w_{1}\,\overline{*}\,w_{2}))\circledast z_{1}^{n})+\sum_{k+l=n \atop k, l\ge 1}
\zeta(d(\phi(w_{1}))\circledast z_{1}^{k})\,
\zeta(d(\phi(w_{2}))\circledast z_{1}^{l})=0 \quad 
(w_{1}, w_{2} \in \mathfrak{h}_{>0}^{1}).
\label{eq:Kawashima-original}
\end{align}
Setting $n=1$ we obtain linear relations for MZV's  
\begin{align}
\zeta(z_{1}\circ d(\phi(w_{1}\,\overline{*}\,w_{2})))=0 \qquad (w_{1}, w_{2} \in \mathfrak{h}_{>0}^{1}).  
\label{eq:Kawashima-linear-MZV}
\end{align}

\subsection{A $q$-analogue of Kawashima's relation}

\begin{prop}
For $w_{1}, w_{2} \in \mathfrak{h}_{>0}^{1}$ and $n\ge 1$, the following relation holds: 
\begin{align}
\zeta_{q}(d_{q}(\phi(w_{1}*_{-}w_{2}))\circledast_{q} z_{1}^{n})+\sum_{k+l=n \atop k, l\ge 1}
\zeta_{q}(d_{q}(\phi(w_{1}))\circledast_{q} z_{1}^{k})\,
\zeta_{q}(d_{q}(\phi(w_{2}))\circledast_{q} z_{1}^{l})=0. 
\label{eq:Kawashima-quad}
\end{align}
\end{prop} 

\begin{proof}
If $w \in \mathfrak{h}_{>0}^{1}$,  
we have $\nabla_{q}(S_{w})(0)=0$ and 
\begin{align*}
v(\nabla_{q}(S_{w})(n))=v(-s_{\phi(w)}(n-1))\ge n \qquad (n>0).  
\end{align*}
{}from Proposition \ref{prod:finite-duality}. 
Hence $\nabla_{q}(S_{w})$ satisfies the condition \eqref{eq:conv-cond}, and 
the series $F_{w}(z):=f_{\nabla(S_{w})}(z)$ is well-defined. 
Expanding $F_{w}(z)$ at $z=1$ using Proposition \ref{prop:Newton-expand} and 
Proposition \ref{prop:prod-sa}, 
we see that 
\begin{align}
F_{w}(z)&=\sum_{m=1}^{\infty}
\left(\sum_{n=1}^{\infty}\nabla(S_{w})(n)\,a_{z_{1}^{m}}(n-1)\right)
\left(\frac{z-1}{1-q}\right)^{m} 
\label{eq:F-expansion} \\ 
&=-\sum_{m=1}^{\infty}
\left(\sum_{n=1}^{\infty}s_{\phi(w)}(n-1)\,a_{z_{1}^{m}}(n-1)\right)
\left(\frac{z-1}{1-q}\right)^{m}  
\nonumber \\ 
&=-\sum_{m=1}^{\infty}
\left(\sum_{n=0}^{\infty}a_{d_{q}(\phi(w))\circledast_{q} z_{1}^{m}}(n)\right)
\left(\frac{z-1}{1-q}\right)^{m}  
\nonumber \\  
&=-\sum_{m=1}^{\infty}
\zeta_{q}(d_{q}(\phi(w))\circledast_{q} z_{1}^{m})
\left(\frac{z-1}{1-q}\right)^{m}.   
\nonumber 
\end{align}

{}From Lemma \ref{prop:prod-S} and Proposition \ref{prop:Newton-prod}, 
we have $F_{w_{1}}F_{w_{2}}=F_{w_{1}*_{-}w_{2}}$ for $w_{1}, w_{2} \in \mathfrak{h}_{>0}^{1}$.  
Substituting \eqref{eq:F-expansion}, we get the relation \eqref{eq:Kawashima-quad}. 
\end{proof}

\begin{cor}
For $w_{1}, w_{2} \in \mathfrak{h}_{>0}^{1}$, we have 
\begin{align*}
\zeta_{q}^{\star}(z_{1}\circ \phi(w_{1}*_{-}w_{2}))=0.  
\end{align*}
\end{cor}

\begin{proof}
Set $n=1$ in \eqref{eq:Kawashima-quad}. Then it follows from 
Lemma \ref{lem:action-commute} and Corollary \ref{cor:to-star}.  
\end{proof}

Let us rewrite the quadratic relation \eqref{eq:Kawashima-quad} 
into a similar form to Kawashima's relation \eqref{eq:Kawashima-original}. 
Consider the map 
\begin{align*}
\Psi:=\phi \, d_{q}^{-1} d \, \phi \, : \mathfrak{h}^{1} \to \mathfrak{h}^{1}.  
\end{align*}
Note that $\Psi(1)=1$. 

\begin{lem}\label{lem:Psi-recursion}
Let $w \in \mathfrak{h}^{1}$. Then 
\begin{align*}
& \Psi(z_{1}w)=z_{1}\Psi(w), \\ 
& \Psi(z_{i}w)=(z_{1}-\hbar z_{0})\circ \Psi(z_{i-1}w) \qquad (i \ge 2).   
\end{align*}
\end{lem}

\begin{proof}
Note that $\phi(z_{i}w)=z_{1}^{i-1}(z_{1}\circ\phi(w))$ for $i\ge 1$ and $w \in \mathfrak{h}^{1}$. 
Since the maps $d_{q}$ and $d$ commute with the action $\circ$ of $\mathfrak{z}$, 
we get $\Psi(z_{1}w)=z_{1}\Psi(w)$. 
To show the second formula, use the identity 
\begin{align*}
(d_{q}^{-1}d)(z_{1}w)=(z_{1}-\hbar z_{0})\circ(d_{q}^{-1}d)(w) \qquad (w \in \mathfrak{h}^{1})
\end{align*} 
following from the definition of $d_{q}$ and $d$. 
\end{proof}

\begin{prop}\label{prop:Psi-recursion}
Let $w \in \mathfrak{h}^{1}$. Then $\Psi(z_{i}w)=\xi_{i}\Psi(w)$ for $i\ge 1$, 
where $\xi_{i} \in \mathfrak{z}$ is given by 
\begin{align*}
\xi_{i}:=\sum_{k=0}^{i-1}\binom{i-1}{k}(-\hbar)^{i-1-k}z_{k+1}.  
\end{align*} 
\end{prop}

\begin{proof}
Proposition \ref{lem:Psi-recursion} implies that it suffices to prove 
$(z_{1}-\hbar z_{0})\circ \xi_{i}=\xi_{i+1}$ for $i \ge 1$. 
This can be checked by direct calculation.  
\end{proof}

\begin{prop}\label{prop:intertwine}
Let $w_{1}, w_{2} \in \mathfrak{h}^{1}$. Then 
\begin{align*}
\Psi(w_{1})*_{-}\Psi(w_{2})=\Psi(w_{1}\,\overline{*}\,\,w_{2}).
\end{align*} 
\end{prop}

\begin{proof}
We can assume that $w_{1}$ and $w_{2}$ are words. 
Let us prove the proposition by induction on the depth of $w_{1}$ and $w_{2}$. 
If $w_{1}=1$ or $w_{2}=1$, it is trivial. 
Set $w_{1}=z_{i}w_{1}'$ and $w_{2}=z_{j}w_{2}'$. 
{}From Proposition \ref{prop:Psi-recursion} and the induction hypothesis, we find 
\begin{align*}
\Psi(w_{1})*_{-}\Psi(w_{2})&=(\xi_{i}\Psi(w_{1}'))*_{-}(\xi_{j}\Psi(w_{2}')) \\ 
&=\xi_{i}\Psi(w_{1}'\,\overline{*}\,w_{2})+\xi_{j}\Psi(w_{1}\,\overline{*}\,w_{2}')+(\xi_{i}\circ_{-}\xi_{j})
\Psi(w_{1}'\,\overline{*}\,w_{2}').   
\end{align*}
By direct calculation we see that $\xi_{i}\circ_{-}\xi_{j}=-\xi_{i+j}$ for $i, j\ge 1$.  
Therefore we get 
\begin{align*}
\Psi(w_{1})*_{-}\Psi(w_{2})&=\xi_{i}\Psi(w_{1}'\,\overline{*}\,w_{2})+
\xi_{j}\Psi(w_{1}\,\overline{*}\,w_{2}')-\xi_{i+j}\Psi(w_{1}'\,\overline{*}\,w_{2}') \\ 
&=\Psi(z_{i}(w_{1}'\,\overline{*}\,w_{2})+z_{j}(w_{1}\,\overline{*}\,w_{2}')-z_{i+j}
(w_{1}'\,\overline{*}\,w_{2}')) \\ 
&=\Psi(w_{1}\,\overline{*}\,w_{2}).   
\end{align*}
\end{proof}

Now we are ready to derive a $q$-analogue of Kawashima's relations: 

\begin{thm}\label{thm:Kawashima-quad-modified}
For $w_{1}, w_{2} \in \mathfrak{h}_{>0}^{1}$ and $n\ge 1$, the following relation holds: 
\begin{align}
\zeta_{q}(d(\phi(w_{1}\,\overline{*}\,\,w_{2}))\circledast_{q} z_{1}^{n})+\sum_{k+l=n \atop k, l\ge 1}
\zeta_{q}(d(\phi(w_{1}))\circledast_{q} z_{1}^{k})\,
\zeta_{q}(d(\phi(w_{2}))\circledast_{q} z_{1}^{l})=0. 
\label{eq:Kawashima-quad-modified}
\end{align}
\end{thm} 

\begin{proof}
Substitute $\Psi(w_{i})$ into $w_{i} \, (i=1, 2)$ in the quadratic relation \eqref{eq:Kawashima-quad}. 
Then we get the relation \eqref{eq:Kawashima-quad-modified} using 
Proposition \ref{prop:intertwine} and $d_{q}\phi\Psi=d \phi$. 
\end{proof}

Setting $n=1$ in \eqref{eq:Kawashima-quad-modified}, we obtain linear relations for $q$MZV's 
in the same form as \eqref{eq:Kawashima-linear-MZV}: 

\begin{cor}\label{cor:Kawashima-linear}
For $w_{1}, w_{2} \in \mathfrak{h}_{>0}^{1}$, we have 
\begin{align*}
\zeta_{q}(z_{1}\circ d(\phi(w_{1}\,\overline{*}\,\,w_{2})))=0.   
\end{align*}
\end{cor}

\bigskip 
\section*{Acknowledgments}

The research of the author is supported by Grant-in-Aid for 
Young Scientists (B) No.\,20740088. 
The author is grateful to Yasuo Ohno, Jun-ichi Okuda  and Tatsushi Tanaka for helpful informations.

\end{document}